\documentclass[3p,times,sort&compress]{ctextemp_elsarticle}
 \UseRawInputEncoding
\usepackage{amsmath,amssymb,amsthm,amsxtra}
\usepackage{graphicx}
\usepackage{amssymb}
\usepackage{color}

\newtheorem{thm}{Theorem}[section]

\newtheorem{cor}[thm]{Corollary}

\newtheorem{lem}[thm]{Lemma}

\newtheorem{rmk}[]{Remark}

\newcommand{\wto}{\to\kern-1.22em\raise.9ex
                  \hbox{\small\rm w}\hskip.7em} 
\newcommand{\wwto}{\to\kern-1.2em\raise.9ex
                  \hbox{\small\rm w}\hskip.7em} 

\journal{arXiv}
\numberwithin{equation}{section}

\def\<{\langle}
\def\>{\rangle}
\def\p{\partial}

\def\e{\epsilon}
\def\ve{\varepsilon}

\def\l{\langle}
\def\r{\rangle}
\def\i{\mathbf i}
\def\be{\begin{equation}}
\def\ee{\end{equation}}
\def\ba{\begin{array}}
\def\ea{\end{array}}

\begin{document}

\begin{frontmatter}

\title{KAM theorem with large perturbation and application to network of Duffing oscillators}

\author{Xiaoping Yuan\corref{cor}}
\ead{xpyuan@fudan.edu.cn}
\author{Lu Chen}
\ead{chenlu@zju.edu.cn}
\author{Jing Li}
\ead{xlijing@sdu.edu.cn}
\cortext[cor]{Corresponding author.}

\begin{abstract}
We prove that there is an invariant torus with given Diophantine frequency vector for  a class of Hamiltonian systems defined by an integrable  large Hamiltonian function with a large non-autonomous Hamiltonian perturbation. As for application, we prove that a finite network of Duffing oscillators with periodic exterior forces possesses Lagrangian stability for almost all initial data.

\end{abstract}

\end{frontmatter}

\section{Introduction}                                                   \label{Introduction}

It is well-known that the forced harmonic oscillator
\begin{equation}\label{eq1-1}
\frac{d^{2}x}{dt^{2}}+\alpha x=k\sin \omega t,\;\alpha>0
\end{equation}
has a periodic solution
$$
\frac{k}{\alpha-\omega^{2}}\sin \omega t,\;\;(\alpha\neq \omega^{2}).
$$
When $\alpha=\omega^{2},$ resonance happens. Duffing \cite{Duffing} introduced a nonlinear oscillator with a cubic stiffness term to
describe the hardening spring effect observed in many mechanical problems
\begin{equation}\label{eq1-2}
\frac{d^{2}x}{dt^{2}}+\alpha x+\beta x^{2}+\gamma x^{3}=k \sin \omega t,\;\;\alpha>0,\;\gamma<0.
\end{equation}
Moon-Holmes \cite{Moon-Holmes1979,Moon-Holmes1980} showed that the equation of the form
\begin{equation}\label{eq1-3}
\frac{d^{2}x}{dt^{2}}+\delta\frac{dx}{dt}-x+x^{3}=k \cos \omega t
\end{equation}
provides the simplest possible model for the forced vibrations of a cantilever beam in the nonuniform field of two permanent
magnets. Such kind of equations are now called Duffing equations (oscillators). After proving his famous twist theorem as a part of
KAM theory, Moser \cite{Moser1962} proposed the study of the Lagrangian stability
for Duffing equation
\begin{equation}\label{eq1-3+}
\frac{d^{2}x}{dt^{2}}+\alpha x+\beta x^{3}=p(t),\;\;\alpha\in \mathbb{R},\;\beta>0,
\end{equation}
where $p(t)$ is a periodic function. The Lagrangian stability means that all solutions exist globally and are bounded in some norm for all time. The proof of the stability for \eqref{eq1-3+} with $\alpha=0, \beta=2$ was completed by Morris \cite{Morris}. Subsequently, the Lagrangian stability result was extended
by Dieckerhoff-Zehnder \cite{Dieckerhoff1987} to a wider class of systems
\begin{equation}\label{eq1-4}
\frac{d^{2}x}{dt^{2}}+x^{2n+1}+\sum_{j=0}^{2n}p_{j}(t)x^{j}=0,\;\;n\geq 1,
\end{equation}
where $p_{j}(t)$'s are smooth enough and  of period, say, $2\pi$. See \cite{LeviCMP,Liu1989,Liu1992,Yuan1995,Yuan1998} for more details.

In many research fields such as physics, mechanics and mathematical biology as so on arise networks of coupled Duffing oscillators of various form.
 For example, the evolution equations for
the voltage variables $V_1$ and $V_2$ obtained using the Kirchhoff's
voltage law are
 \begin{equation}\label{eq1-14}
\left\{ \begin{array}{ll}
  R^2 C^2 \frac{d^2 V_1}{d t^2}=-(\frac{R^2 C}{R_1})\frac{d V_1}{dt}-(\frac{R}{R_2})V_1-(\frac{R}{100 R_3}) V_1^3+(\frac{R}{R_C})V_2+f \sin \omega t,\\
    R^2 C^2 \frac{d^2 V_2}{d t^2}=-(\frac{R^2 C}{R_1})\frac{d V_2}{dt}-(\frac{R}{R_2})V_2-(\frac{R}{100 R_3}) V_2^3+(\frac{R}{R_C})V_1 \end{array}
\right.
\end{equation}
 where $R$'s and $C$'s are resistors and capacitors, respectively.  This equation can be regarded as one  coupled by two Duffing oscillators. See  \cite{Jothimurugan2016,Kiss,Kovaleva,Clerc2018,Sarkar,Chatterjee,Shena,Sarma} for more details.

In the present paper, we concern the Lagrangian stability for coupled Hamiltonian system of $m$ Duffing oscillators: 
\begin{equation}\label{eq1-15}
\ddot{x_{i}}+x_{i}^{2n+1}+\frac{\partial F}{\partial x_{i}}=0,\;\;i=1, 2, \cdots, m,
\end{equation}
where the polynomial potential $F=F(x, t)=\sum_{\alpha\in \mathbb{N}^{m}, |\alpha|\leq 2n+1}p_{\alpha}(t)x^{\alpha},$ $x\in \mathbb{R}^{m}$ with $p_{\alpha}(t)$ is of period $2 \pi$, and $n$ is a given natural number. When $m=1$, \eqref{eq1-15} is reduced to \eqref{eq1-4}. Note that \eqref{eq1-15} is a Hamiltonian system with freedom $m+1/2$ where time takes $1/2$-freedom. In generally speaking, the study of the stability of \eqref{eq1-15} with $m>1$ is a difficult task. In order to see this point, let us recall the method used by Dieckerhoff-Zehnder \cite{Dieckerhoff1987} for the case $m=1$.  The key ingredient in \cite{Dieckerhoff1987} is
the Moser's twist theorem \cite{Moser1962}  which shows that the twist mapping
\begin{equation}\label{eq1-5}
P_{\varepsilon}: \left\{
                     \begin{array}{ll}
                       \rho_{1}=\rho+\varepsilon g(\rho, \theta), & \rho\in [a, b] ,\\
                       \theta_{1}=\theta+\alpha(\rho)+\varepsilon f(\rho, \theta), & \theta\in \mathbb{T}=\mathbb{R}/2 \pi
                     \end{array}
                   \right.
\end{equation}
possesses an invariant curve with rotational frequency $\omega=\alpha(\rho_{0})\in$ D.C. (D. C. means Diophantine conditions) provided that
$\parallel f\parallel_{C^{333}},$ $\parallel g\parallel_{C^{333}}\lesssim 1$, $0<\ve \ll 1$ and $P_{\varepsilon}$ is symplectic or preserving-intersection.
($C^{333}$ can be relaxed to $C^{3+\mu}$ with $\mu>0$. See R\"{u}ssman \cite{Russmann}.) Observe that the mapping $P_\ve$ is an integrable one attached by a small perturbation $\ve (f,g)$.
Although $\ddot{x}+x^{2n+1}=0$ is indeed
integrable,  the perturbation $\sum_{j=0}^{2n}p_{j}(t)x^{j}$ is not small at all in \eqref{eq1-4}. An important observation by Dieckerhoff-Zehnder \cite{Dieckerhoff1987} is that the perturbation $\sum_{j=0}^{2n}p_{j}(t)x^{j}$ is relatively small with respect to the integrable part $x^{2n+1}$ in the neighborhood of the infinity.
 More exactly, one can write \eqref{eq1-4} as a Hamiltonian equation with Hamiltonian $H$:
$$H=H_{0}(I)+R(I, \theta, t),\;\;(I,\theta,t)\in \mathbb {R}\times \mathbb {T}\times \mathbb {T},$$ where $(I, \theta)$ are action-angle variables, and
\begin{equation}\label{eq1-6-}
H_{0}(I)\sim I^{\frac{2n+2}{n+2}},\; R(I, \theta, t)\sim I^{\frac{2n+1}{n+2}}.
\end{equation}
Note
\begin{equation}\label{eq1-6}
\lim_{I\rightarrow \infty} (H_{0}(I))^{-1} R(I, \theta, t)=0.
\end{equation}
If the integrable Hamiltonian $H_0$ and the perturbation $R$ obey \eqref{eq1-6}, we call that $R$  is
relatively small with respect to  $H_0$ in the neighborhood of the infinity.
Then by a series of symplectic transformations, the relatively small $R$ can be changed into a truly small perturbation.
The symplectic transformations can be implicitly defined by
\begin{equation}\label{eq1-7}
\Psi: \left\{
        \begin{array}{ll}
          I=\mu+\frac{\partial S}{\partial \theta}, \\
          \phi=\theta+\frac{\partial S}{\partial \mu},
        \end{array}
      \right.
\end{equation}
where $S=S(\mu, \theta, t)$ is generating function which obeys the homological equation
\begin{equation}\label{eq1-8}
H'_{0}(\mu)\frac{\partial S}{\partial \theta}+R(\mu, \theta, t)=[R](\mu, t),
\end{equation}
where $[R](\mu, t)=\frac{1}{2\pi}\int_{0}^{2\pi}R(\mu, \theta, t)d\theta.$ When $m=1$, the homological equation is a scalar equation. Thus,
\begin{equation}\label{eq1-9}
S=\int_{0}^{\theta}\frac{[R](\mu, t)-R(\mu, \theta, t)}{H'_{0}(\mu)}d\theta.
\end{equation}
By \eqref{eq1-6-}, we have
\begin{equation}\label{eq1-10}
S\sim I^{\frac{n+1}{n+2}}.
\end{equation}
Now the perturbation $R$ is changed into  $R_{1}$:
\begin{equation}\label{eq1-11}
R_{1}=\frac{\partial S}{\partial t}+\int_{0}^{1}(1-\tau)H''_{0}(\mu+\tau\frac{\partial S}{\partial \theta})(\frac{\partial S}{\partial \theta})^{2}d\tau
\sim O(I^{\frac{n+1}{n+2}})+O(I^{\frac{2n}{n+2}})=O(I^{\frac{2n}{n+2}}).
\end{equation}
Repeat the procedure as the above $\nu$ times with $\nu> 1+\frac{4}{n}+\log_{2}{n}$.  Then $R$ is changed into
\begin{equation}\label{eq1-12}
R_{\nu}=O(I^{-\frac{c}{n+2}}),\;\;c>0.
\end{equation}
Now $R_{\nu}$ is small when $|I|\rightarrow +\infty.$ By Moser's twist theorem, one can prove that there are many invariant cylinder of time period $2\pi$ around $\infty$
in the extended phase $(x, \dot{x}, t).$ The solution $(x(t),\dot x(t))$ with the initial data $(x(0),\dot x(0))$ in the cylinder is confined in the cylinder.
Thus, the Duffing oscillator \eqref{eq1-4} has Lagrange stability
\begin{equation}\label{eq1-13}
\sup_{t\in \mathbb{R}}(|x(t)|+|\dot{x}(t)|)<\infty.
\end{equation}

When $m\ge 2,$ equation \eqref{eq1-9} does not hold any more. In order to see this clearly, passing \eqref{eq1-8} to Fourier coefficients of $S$ and $R,$ we have
\begin{equation}\label{eq1-16}
\widehat{S}(\mu, k, t)= \frac{{\mathbf i}}{ \langle \p_{\mu}\, H_0(\mu), k\rangle}\widehat{R}(\mu, k, t),\;\;k\in \mathbb{Z}^{m}\setminus \{0 \}, \, {\mathbf i}^2=-1,
\end{equation}
where $\widehat{S}(\mu, k, t)$ and $\widehat{R}(\mu, k, t)$ are Fourier coefficients of $S(\mu, \cdot, t)$ and $R(\mu, \cdot, t),$ respectively. Note that $0$ is in the closure of the set $\{\langle \p_{\mu}\, H_0(\mu), k\rangle:\;  k\in \mathbb{Z}^{m}\setminus \{0 \}\}$, that is, $\langle \p_{\mu}\, H_0(\mu), k\rangle$'s are the notorious small divisors. Thus the estimate \eqref{eq1-10} does not hold when $m\ge 2$.

 In the present paper, we directly construct a KAM theorem to deal Hamiltonian system with large perturbation.

In order to state the theorem, we need some notations. By $C$ (or $c$) denote a universal positive  constant which may be different in different places. When those constants $C$ and $c$ are necessarily distinguished, denote them by $C_0,C_1,c_0,c_1$, etc.  Let positive integer $d$ be the freedom of the to-be considered Hamiltonian.  Let $[1,2]^{d}$ be the product of $d$ intervals $[1,2]$ and $\mathbb T^{d+1}=\mathbb R^{d+1}/(2\pi \mathbb Z)^{d+1}$. Denote by $\ve>0$ a small constant which measures the size of Hamiltonian functions.

\begin{thm}\label{thm01} Consider a Hamiltonian $H=\ve^{-a}H_0(I)+\ve^{-b} R(\theta,t,I)$ where $a,b$ are given positive constants with $a>b$, and $H_0$ and $R$ obey the following conditions:
\begin{enumerate}
\item $H_0$ and $R(\theta,t,I)$ are real analytic functions in the domain $\mathbb T^{d+1}\times [1,2]^d$ and
\[||H_0||:=\sup_{I\in[1,2]^d}|H_0(I)|\le c_1, \; ||R||:=\sup_{(\theta,t,I)\in\mathbb T^{d+1}\times [1,2]^d}|R(\theta,t,I)|\le c_2\]
\item $H_0$ is non-degenerate in Kolmogorov's sense:
\[\text{det}\,\left(\p_I^2\, H_0(I) \right)\ge c_3>0,\forall\; I\in [1,2]^d.\] \end{enumerate}
Then there exists $\e^*=\e^*(a,b,d, c_1,c_2,c_3)\ll 1$ such that for any $\ve$ with $|\ve|<\e^*$,  the Hamiltonian system
$$\dot \theta=\frac{\p H(\theta,t,I)}{\p I},\;\dot I=-\frac{\p H(\theta,t,I)}{\p \theta}$$ possesses a $d+1$ dimensional invariant torus of rotational frequency vector $(\omega(I_0),2\pi)$ with $\omega(I):=\frac{\p H_0(I)}{\p I}$, for any $I_0\in  [1,2]^d$  and $\omega(I_0)$ obeying   Diophantine conditions:
    \[|\ve^{-a}\l k,\omega(I_0)\r +l|\ge \frac{\ve^{-a}\gamma}{|k|^{d+1}}>\frac{\gamma}{|k|^{d+1}},\; \gamma=(\log\frac{1}{\ve})^{-2C_1}, k\in\mathbb Z^d\setminus\{0\},l\in\mathbb Z, |k|+|l|\le (\log\frac{1}{\ve})^{C_1},\] and
     \[|\ve^{-a}\l k,\omega(I_0)\r +l|\ge \frac{\gamma}{(1+|k|)^{d+1}},\;  (k,l)\in\mathbb Z^d\times \mathbb Z, |k|+|l|> (\log\frac{1}{\ve})^{C_1}.\]
\end{thm}
Applying Theorem \ref{thm01} to \eqref{eq1-15} we have the following theorem.

\begin{thm}\label{thm02} The equation \eqref{eq1-15} is almost Lagrangian stable, that is, for almost all initial data $(x_j(0),\dot x_j(0): j=1,...,m)$, the solution $(x_j(t),\dot x_j(t): j=1,...,m)$ exists globally and obeys
\be\label{Lag2}\sup_{t\in\mathbb R}\sum_{i=1}^m |x_i(t)|+|\dot x_i(t)|<C\ee
where $C$ depends on the initial data $(x_j(0),\dot x_j(0): j=1,...,m)$.
\end{thm}
\begin{rmk} Actually we can prove that  the solutions to  \eqref{eq1-15} are time quasi-periodic for almost all large initial data. See $\S$ \ref{section5}.
\end{rmk}
\begin{rmk} Theorems \ref{thm01} and \ref{thm02} hold true for the reversible systems, respectively. We do not pursue this end here.
\end{rmk}
\begin{rmk}\label{RMK2} Let $\varTheta=\{I_0\in[1,2]^d:\; \omega(I_0) \,\text{obeys the Diophantine conditions}\}$. We claim that the Lebesgue measure of $\varTheta$ approaches to $1$:
\[\text{Leb} \varTheta\ge 1-C (\log\frac{1}{\ve})^{-C}\to 1,\;\text{as}\; \ve\to 0.\]
Let
\[\tilde\Theta_{k,l}=\left\{\xi\in\omega([1,2]^d):\; |\ve^{-a}\l k,\xi\r +l|\le \frac{\ve^{-a}\gamma}{|k|^{d+1}},\; 0\neq k\in\mathbb Z^d,\, l\in \mathbb Z, |k|+|l|\le (\log\frac{1}{\ve})^{C_1}\right\}\] and
\[\tilde\Theta_{k,l}=\left\{\xi\in\omega([1,2]^d):\; |\ve^{-a}\l k,\xi\r +l|\le \frac{\gamma}{(1+|k|)^{d+1}},\; (k,l)\in\mathbb Z^d\times \mathbb Z, |k|+|l|> (\log\frac{1}{\ve})^{C_1}\right\}.\] Note that there exists a direction $k(\xi)$ such that
\[\p_{k(\xi)}(\ve^{-a}\l k,\xi\r +l)=\ve^{-a}|k|,\; k=(k_1,...,k_d)\in\mathbb Z^d\setminus\{0\}.\]
Thus,
\[\text{Leb}\tilde\Theta_{k,l}\le C \gamma/|k|^{d+2},\; k\neq 0, |k|+|l|\le (\log\frac1{\ve})^{C_1},\]
\[\text{Leb}\tilde\Theta_{k,l}\le C \ve^{a}\gamma/|k|^{d+2},\; k\neq 0, |k|+|l|>(\log\frac1{\ve})^{C_1}.\]
 Also note that $\tilde\Theta_{k,l}=\emptyset$ when $k=0,l\neq 0$ and when $|l|>1+\ve^{-a}|k|\sup_{I\in[1,2]^d}|\omega(I)|$. Therefore,
\[\text{Leb}\,\left(\bigcup_{(0,0)\neq (k,l)\in\mathbb Z^{d+1}}\tilde\Theta_{k,l}\right)\le C \gamma (\log\frac{1}{\ve})^{C_1}\le C(\log\frac{1}{\ve})^{-C_1}.\] Let $\Theta=[1,2]^d\setminus \left(\bigcup_{(0,0)\neq (k,l)\in\mathbb Z^{d+1}}\omega^{-1}(\tilde\Theta_{k,l})\right)$. By the Kolmogorov's non-degenerate condition, the map $\omega:[1,2]^d\to \omega([1,2]^d)$ is a diffeomorphism in both direction. Then the proof of the claim is completed by letting  $\Theta=[1,2]^d\setminus \left(\bigcup_{(k,l)\in\mathbb Z^{d+1}\setminus\{0,0\}}\omega^{-1}(\tilde\Theta_{k,l})\right)$.
\end{rmk}

\section{Normal Form}
By the compactness of $\mathbb T^{d+1}\times [1,2]^d$, there is a constant $s_0>0$  such that the Hamiltonian functions $H_0(I), R(\theta,t,I)$ are analytic in the complex neighborhood $\mathbb T^{d+1}_{2s_0}\times [1,2]_{2s_0}^d$ of  $\mathbb T^{d+1}\times [1,2]^d$ and are real for real arguments where
\[\mathbb T^{d+1}_{2s_0}\times [1,2]_{2s_0}^d=\{\phi\in\mathbb C^{d+1}/(2\pi\mathbb Z)^{d+1}: \; |\Im \phi|\le 2s_0\}
\times\{z\in\mathbb C^d:\;\text{dist}_{\mathbb C}\,(z,\mathbb [1,2]^d)\le 2s_0 \}.\] We will call that a function of complex variables is real analytic if it is analytic on the complex variables and real for the real arguments. Thus $H_0$ and $R$ are real analytic and we can assume
\be\label{yy1}||\p_I^\alpha H_0||_{2s_0}=\sup_{I\in [1,2]^d_{2s_0}}|\p_I^\alpha H_0(I)|\le C,
||\p^\alpha_{I}\p^\beta_{\phi}\, R(\phi,I)||_{2s_0}=\sup_{(\phi,I)\in\mathbb T^{d+1}_{2s_0}\times [1,2]_{2s_0}^d}
|\p^\alpha_{I}\p^\beta_{\phi} R(\phi,I)|\le C,\,|\alpha|+|\beta|\le 2.\ee
Now take $I_0\in [1,2]^d$ such that $\omega(I_0)=\p_I H_0(I_0)$ obeys D.C.:
\be\label{yy2}|\ve^{-a}\l k, \omega(I_0)\r+l|\ge \frac{\ve^{-a}\gamma}{|k|^{d+1}},\; |k|+|l|\le {(\log\frac{1}{\ve})^{C_1}},\,
k\in\mathbb Z^{d}\setminus\{0\},\, l\in\mathbb Z.\ee
Take a constant $C_2$ with $C_2>5(d+1)C_1$ and define a neighborhood of $I_0$
\[B_0=\{I=I_0+z:\; z\in\mathbb C^d,\, |z|\le (\log\frac{1}{\ve})^{-C_2}\}.\]
By \eqref{yy1} and \eqref{yy2}, we have that for $|k|+|l|\le (\log\frac{1}{\ve})^{C_1},\,
k\in\mathbb Z^{d}\setminus\{0\},\, l\in\mathbb Z,\, I\in B_0$, the frequency $\omega(I)$ obeys
\be\label{yy3} |\ve^{-a}\l k, \omega(I)\r+l|\ge |\ve^{-a}\l k, \omega(I_0)\r+l|-\ve^{-a}|\omega(I)-\omega(I_0)||k|
 \ge \frac{\ve^{-a}\gamma}{2|k|^{d+1}}. \ee
Let $\Omega_0=\frac12\p^2_IH_0(I_0)$. By the Kolmogorov's non-degeneracy, we have that $\Omega_0$ is invertible matrix and
\be\label{yy4} ||\Omega_0||\le C,\; ||\Omega^{-1}_0||\le C.\ee

Introduce a truncation operator $\Gamma=\Gamma_{K}$ depending on $K>0$ as follows:
For any function $f:\mathbb{T}^{d+1}\rightarrow \mathbb{C}$  (or $\mathbb{C}^{l}, l\geq 1$), write $f(x)=\sum_{k\in \mathbb{Z}^{d+1}}\widehat{f}(k)e^{i\langle k, x\rangle}$ and define function $(\Gamma_{K}f)(x)$:
 $$(\Gamma_{K}f)(x)=\sum_{|k|\leq K}\widehat{f}(k)e^{\i\langle k, x\rangle}.$$
Now consider $H(\theta, t, I)=\varepsilon^{-a}H_0(I)+\varepsilon ^{-b}R(\theta, t, I),$
where $(\theta, t)\in\mathbb{{T}}^{d+1}_{2s_{0}}$ and $I\in B_0$. Then $H_0(I)$ and
$R(\theta, t, I)$ are real analytic on $\mathbb{{T}}^{d+1}_{2s_{0}}\times B_0$.
Take $x=(\theta, t).$ Decompose $R(\theta, t, I)=\Gamma_{K}R+(1-\Gamma_{K})R,$
where $$(\Gamma_{K}R)(x, I)=\sum_{|k|+|l|\leq K}\widehat{R}(k, l, I)e^{\i(\langle k, \theta\rangle+l\,t)},\;\;(k, l)\in \mathbb{Z}^{d}\times \mathbb{Z}.$$
Let $A=200d(a+b)$ and $m_0=2+\left[\frac{b+A}{a-b}\right]$ where $[\cdot]$ is the integer part of a positive number.
 Define sequences
\begin{itemize}
\item $ s_j=2^{-j}s_0,\; s_j^{l}=s_{j}-\frac{l}{10}(s_j-s_{j+1}),l=0,1,...,10, j=0,1,2,...,m_0;$
\item $K_j=\frac{A}{s_j-s_j^1}\log\frac{1}{\ve},\;j=0,1,2,...,m_0;$
\item $\tau_0=(\log\frac{1}{\ve})^{-C_2},\, \tau_j=2^{-j}\tau_0,\tau_j^{l}=\tau_{j}-\frac{l}{10}(\tau_j-\tau_{j+1}),l=0,1,...,10,\; j=0,1,...,m_0;$
\item $B(\tau_0)=B_0, B(\tau_j)=\{z\in\mathbb C^d: \, |z-I_0|\le \tau_j\},\;  j=0,1,...,m_0$. \end{itemize}

Let $C_1>1$. We have
\be\label{yy5}K_j\le K_{m_0}<(\log\frac{1}{\ve})^{C_1},\;j=0,1,2,...,m_0.\ee
For a function $f$ defined in $\mathbb T_{s}^{d+1}\times B(\tau)$ or $B(\tau)$ , define
\[||f||_{s,\tau}=\sup_{(\phi,I)\in\mathbb T_{s}^{d+1}\times B(\tau)}|f(\phi,I)|,\; \text{or}\; ||f||_{\tau}=\sup_{I\in B(\tau)}|f(I)|.\]
 By Cauchy's formula, we have
\begin{eqnarray*}
\parallel (1-\Gamma_{K_0})R\parallel_{s_{0}, \tau_{0}}&=&\parallel \sum_{|k|+|l|> K_{0}} \widehat{R}(k, l, I)e^{\i(\langle k, \theta\rangle+l\cdot t)}\parallel_{s_0,\tau_0}
\leq (\sum_{|k|+|l|> K_0} e^{-(2s_{0}-s_{0})(|k|+|l|)})\parallel R \parallel_{2s_{0}, 2s_0}\\
&\leq & \frac{C}{(s_0)^{d+1}}e^{-s_0\,K_{0}}\leq C\ve^{A},
\end{eqnarray*}
and
$$\parallel \Gamma_{K_0}R\parallel_{s_{0}, \tau_{0}}\leq C.$$
Let $$H^{(0)}=H,\;\;R^{(0)}=\varepsilon^{-b}\Gamma_{K_0}R(\theta, t, I),\;\;R^{(0)}_+=\ve^{-b}(1-\Gamma_{K_0})R(\theta, t, I).$$
Then
\begin{equation}\label{eq10}
H^{(0)}=\varepsilon^{-a}H_{0}(I)+ R^{(0)}(\theta, t, I)+R^{(0)}_+(\theta, t, I),\;(\theta, t, I)\in \mathbb{T}^{d+1}_{s_0}\times B(\tau_0)
\end{equation}
where $R^{(0)}$, $R_{+}^{(0)}$ are real analytic in $\mathbb{T}^{d+1}_{s_{0}}\times B(\tau_0),$ and
\begin{equation}\label{eq11}
\parallel R^{(0)}\parallel_{s_{0}, \tau_{0}}\leq C\ve^{-b}, \Gamma_{K_0}\, R^{(0)}\equiv  R^{(0)},
\end{equation}
\begin{equation}\label{eq12}
\parallel R_+^{(0)}\parallel_{s_{0}, \tau_{0}}\leq C\ve^{A-b}.
\end{equation}
In order to simplify notation, the universal constants $C$'s in the following lemma may depend on $m_0,C_1,C_2$ , thus, on $a, b, A,C_1,C_2$.
\begin{lem}[Iterative lemma in finite steps]\label{lemma2.1}
Assume that we have a Hamiltonian
\be\label{yy7} H^{(j)}=\ve^{-a}H_0(I)+h^{(j)}(t,I)+R^{(j)}(\theta,t,I)+R^{(j)}_+(\theta,t,I),\, (\theta,t,I)\in\mathbb T^{d+1}_{s_j}\times B(\tau_j),
\,0\le j\le m_0-1,\ee
where
\begin{description}
\item[1(j)] The functions $h^{(j)}(t,I),R^{(j)}(\theta,t,I)$ and $R^{(j)}_+(\theta,t,I)$ are real analytic in $T^{d+1}_{s_j}\times B(\tau_j)$;
\item[2(j)]     \[ h^{(0)}\equiv0,\; ||h^{(j)}||_{s_j,\tau_j}\le C \ve^{-b}, 0\le j\le m_0-1;\]
\item[3(j)] \[||R^{(j)}||_{s_j,\tau_j}\le C \ve^{-b+j(a-b)}(\log\frac{1}{\ve})^{C},\; {\Gamma_{K_j}R^{(j)}=R^{(j)}};\]
\item[4(j)] \[||R^{(j)}_+||_{s_j,\tau_j}\le C \ve^{-b+A}(\log\frac{1}{\ve})^{C}.\]
\end{description}
Then there exists a symplectic coordinate change
\be\label{yy8}\Psi_j:\;  T^{d+1}_{s_{j+1}}\times B(\tau_{j+1})\to T^{d+1}_{s_j}\times B(\tau_j)\ee
with
\be\label{yy9} ||\Psi_j-id||_{s_{j+1},\tau_{j+1}}\le C\ve^{(j+1)(a-b)} (\log\frac{1}{\ve})^{C}\ee
such that
\be\label{1110} H^{(j+1)}=H^{(j)}\circ \Psi_j=\ve^{-a}H_0(I)+h^{(j+1)}(t,I)+R^{(j+1)}(\theta,t,I)+R^{(j+1)}_+(\theta,t,I),\, (\theta,t,I)\in\mathbb T^{d+1}_{s_{j+1}}\times B(\tau_{j+1}),\ee
where the new Hamiltonian functions $h^{(j+1)}(t,I), R^{(j+1)}(\theta,t,I)$ and $R^{(j+1)}_+(\theta,t,I)$ obey the conditions {\bf 1(j+1)-4(j+1)}.
\end{lem}

\begin{proof}
Assume that the change $\Psi_j$ is implicitly defined by
\begin{equation}\label{eq13}
\Psi_j: \left\{
            \begin{array}{ll}
              I=\rho+\frac{\partial S}{\partial \theta}, \\
              \phi=\theta+\frac{\partial S}{\partial \rho},\\ t=t,
            \end{array}
          \right.
\end{equation}
where $S=S(\theta, t, \rho)$ is the generating function, which will be proved to be analytic in a smaller domain
 $\mathbb{{T}}^{d+1}_{s_{j+1}}\times B(\tau_{j+1}).$
By a simple computation, we have
$$d I\wedge d\theta=d\rho \wedge d\theta+\sum_{i,j=1}^d\frac{\partial^{2}S}{\partial\rho_{i}\partial\theta_{j}}d\rho_{i}\wedge d\theta_{j}=d\rho \wedge d\phi.$$
Thus the coordinates change $\Psi_j$ is symplectic if it exists. Moreover, we get the changed Hamiltonian
\begin{equation}\label{eq14}
H^{(j+1)}:=H^{(j)}\circ \Psi_j=\ve^{-a}H_{0}(\rho+\frac{\partial S}{\partial \theta})+h^{(j)}(t,\rho+\frac{\partial S}{\partial \theta})+R^{(j)}(\theta, t, \rho+\frac{\partial S}{\partial \theta})
+R_{+}^{(j)}(\theta, t, \rho+\frac{\partial S}{\partial \theta})+\frac{\partial S}{\partial t},
\end{equation}
where $\theta=\theta(\phi, t, \rho)$ is implicitly defined by \eqref{eq13}. We drop the sup-index $(j)$. We replace $h^{(j)}$ by $h$, for example.

By Taylor formula, we have
\begin{equation}\label{eq15}
H^{(j+1)}=\ve^{-a}H_{0}(\rho)+h(t,\rho)+\ve^{-a}\langle \omega(\rho), \frac{\partial S}{\partial \theta}\rangle
+\frac{\partial S}{\partial t}+R(\theta, t, \rho)+R_{*}+R_{**},
\end{equation}
where
\begin{eqnarray}
R_{**}&=&R_{+} (\theta, t, \rho+\frac{\partial S}{\partial \theta}),\label{eq15+}\\
R_{*}&=&\ve^{-a}\int_{0}^{1}(1-\tau)\partial^{2}_{I}H_{0}(\rho+\tau\frac{\partial S}{\partial \theta}) (\frac{\partial S}{\partial \theta})^{2}d\tau\label{eq16}\\
&&+\int_{0}^{1}\partial_{I}R(\theta, t, \rho+\tau\frac{\partial S}{\partial\theta})\frac{\partial S}{\partial \theta}d\tau\label{eq17}\\
&&+\int_0^1\p_I\, h(t,\rho+\tau\frac{\partial S}{\partial\theta})\frac{\partial S}{\partial \theta}\, d\tau.
\label{eq18}
\end{eqnarray}
Thus, we derive the homological equation:
\begin{equation}\label{eq19}
\frac{\partial S}{\partial t}+\ve^{-a}\langle \omega(\rho), \frac{\partial S}{\partial\theta}\rangle
=\widehat{ R}(0, t, \rho)-R(\theta, t, \rho),\quad S=S(\theta,t,\rho),
\end{equation} where $\widehat{ R}(0, t, \rho)$ is $0$-Fourier coefficient of $R(\theta,t,\rho)$ as the function of $\theta$.
Recalling {\bf 3(j)}, we have $\Gamma_{K_j}R^{(j)}=R^{(j)}$. Thus we can assume $\Gamma_{K_j}S=S$.
For $f\in\{S, R\}$, let
\begin{equation}\label{eq20}
f(\theta,t,\rho)=\sum_{|k|+|l|\leq K_j,k\neq 0}\widehat{f}(k, l, \rho)e^{\i(\langle k, \theta\rangle+lt)},\;\i^2=-1.
\end{equation}
By passing to Fourier coefficients, we have
\begin{equation}\label{eq21}
\widehat{S}(k, l, \rho)=\frac{\i}{\ve^{-a}\langle k, \omega(\rho)\rangle +l}\widehat{R}(k, l, \rho),\;|k|+|l|\leq K_j,\; k\in \mathbb{Z}^{d}\setminus\{0\}, l\in \mathbb{Z}.
\end{equation}
Let $\rho\in B(\tau_j)\subset B(\tau_0)=B_0$.
By \eqref{yy3}, we have that
$$|\widehat{S}(k, l, \rho)|\leq C\frac{\varepsilon^{a}|k|^{d+1}}{\gamma}|\widehat{R}(k, l, \rho)|.$$
Moreover, by {\bf 3(j)}, for $(\theta, t,\rho)\in \mathbb{T}_{s_j^1}^{d+1}\times B(\tau_j),$ we get
\begin{equation}\label{eq35}
|S(\theta, t, \rho)|\leq \frac{C\varepsilon^{a}}{\gamma}\parallel R\parallel_{s_{j}, \tau_j}\sum_{|k|+|l|\leq K_j}|k|^{d+1}e^{-(s_{j}-s_{j}^1)(|k|+|l|)}
\leq C \ve^{(j+1)(a-b)}(\log\frac1{\ve})^{C},
\end{equation}
where $C=C(m_0, d)$.
Then by the Cauchy's estimate, we have
\begin{equation}\label{eq36}
 {||\frac{\partial^{p+q}}{\partial \theta^{p}\partial \rho^{q}}S(\theta, t, \rho)||_{s_j^2,\tau_j^1}\leq C \ve^{(j+1)(a-b)}(\log\frac1{\ve})^{C},
p+q\leq 2, p\geq0, q\geq0}.
\end{equation}
By \eqref{eq13} and \eqref{eq36} and the implicit function theorem (analytic version), we get that there are analytic functions
$u=u(\phi, t, \rho), v=v(\phi, t, \rho)$ defined on
the domain $\mathbb T^{d+1}_{s_{j}^3}\times B(\tau_j^3)\supset \mathbb T^{d+1}_{s_{j+1}}\times B(\tau_{j+1})$ with
\be\label{yy31} \frac{\p S(\theta,t,\rho)}{\p\theta}=u(\phi,t,\rho),\;  \frac{\p S(\theta,t,\rho)}{\p\rho}=-v(\phi,t,\rho).\ee and
\be\label{yy20}||u||_{s_j^3,\tau_j^3}\le \ve^{(j+1)(a-b)}(\log\frac1{\ve})^{C},\;
||v||_{s_j^3,\tau_j^3}\le \ve^{(j+1)(a-b)}(\log\frac1{\ve})^{C}\ee
such that
\begin{equation}\label{eq37}
\Psi_j: \left\{
            \begin{array}{ll}
              I=\rho+u(\phi, t, \rho),\\
              \theta=\phi+v(\phi, t, \rho), \\ t=t.
            \end{array}
          \right.
\end{equation}
Note
\[ \ve^{(j+1)(a-b)}(\log\frac1{\ve})^{C}\ll s_j-s_{j+1},\; \ve^{(j+1)(a-b)}(\log\frac1{\ve})^{C}\ll \tau_j-\tau_{j+1}.\] Then
$\Psi_j(\mathbb T^{d+1}_{s_{j+1}}\times B(\tau_{j+1}))\subset \mathbb T^{d+1}_{s_{j}}\times B(\tau_{j})$. This proves \eqref{yy8} and \eqref{yy9}.
 Then by {\bf 4(j)}, we have
\be\label{yy50}\parallel R_{**}(\phi, t, \rho)\parallel_{s_j^{4}, \tau_j^4}\leq \parallel R_{+}\parallel_{s_{j}, \tau_{j}}\leq  C \ve^{-b+A}(\log\frac{1}{\ve})^{C}.\ee
By \eqref{yy31}, \eqref{yy20}, {\bf 2(j)} and {\bf 3(j)}, we have
\be\label{yy55} ||R_*(\phi, t, \rho)||_{s_j^4,\tau_j^4}\le C\left(\ve^{-a}
 ||u||_{s_j^3,\tau_j^3}^2+||R||_{s_j,\tau_j}||u||_{s_j^3,\tau_j^3}\tau_j^{-1}+||h||_{s_j,\tau_j}||u||_{s_j^3,\tau_j^3} \tau_j^{-1}
  \right)\le C \ve^{-b+(j+1)(a-b)}(\log\frac1{\ve})^{C}. \ee
Let
\be\label{yy60-1} R^{(j+1)}(\phi, t, \rho)=\Gamma_{K_{j+1}} R_*\ee and
\be\label{yy60}
 R^{(j+1)}_+(\phi, t, \rho)=(1-\Gamma_{K_{j+1}})R_*+R_{**}.
\ee
By \eqref{yy50} and \eqref{yy55}, we  see that $R^{(j+1)}$ and $R^{(j+1)}_+$ satisfy {\bf 3(j+1)} and {\bf 4(j+1)}, respectively.
Let
\begin{equation}\label{eq40}
h^{(j+1)}(t,\rho)=h(t,\rho)+\widehat{R}(0,t,\rho).
\end{equation}
Then by {\bf 2(j)} and {\bf 3(j)}, we have
\be\label{67} ||h^{(j+1)}||_{s_{j+1},\tau_{j+1}}\le ||h^j||_{s_j,\tau_j}+||R||_{s_j,\tau_j}
\le C\ve^{-b},\ee
which fulfills {\bf 2(j+1)}. Finally, it is obvious that the functions $h^{(j+1)}(t,I),R^{(j+1)}(\theta,t,I)$ and $R^{(j+1)}_+(\theta,t,I)$ are
analytic in $T^{d+1}_{s_{j+1}}\times B(\tau_{j+1})$. Since $R$ is real for the real arguments, we have
\[\widehat{R}(-k,-l,\rho)=\overline{\widehat{R}(k,l,\rho)}, \;\rho\in\mathbb R^d, (k,l)\neq (0,0)\] where the bar is the complex conjugate. By \eqref{eq21},
we have
\[\widehat{S}(-k,-l,\rho)=\overline{\widehat{S}(k,l,\rho)}, \;\rho\in\mathbb R^d, (k,l)\neq (0,0).\] It follows that $S(\theta,t,\rho)$ is real
for the real arguments $(\theta,t,\rho)$. Moreover, the functions $h^{(j+1)}(t,I),R^{(j+1)}(\theta,t,I)$ and $R^{(j+1)}_+(\theta,t,I)$ are real for real arguments.
By \eqref{eq15}-\eqref{eq18}, \eqref{yy60-1}-\eqref{eq40}, we have that \eqref{1110} holds.
This completes the proof of the lemma.\end{proof}

Applying Lemma \ref{lemma2.1} to \eqref{eq10} with $j=0$ and $h^{(0)}\equiv 0$, we get $m_0$ symplectic changes $\Psi_j$ ($j=0,...,m_0-1$) such that
\be \label{0402-1}\Psi:=\Psi_0\circ\cdots\circ\Psi_{m_0-1},\; \Psi\left(\mathbb T^{d+1}_{s_{m_0}}\times B(\tau_{m_0})\right)\subset \mathbb T^{d+1}_{s_0}\times B(\tau_0),\ee
\be \label{0402-2}H^{(m_0)}=\ve^{-a}H_0(\rho)+h^{(m_0)}(t,\rho)+\tilde R(\phi,t,\rho),\,\, (\phi,t,\rho)\in\mathbb T^{d+1}_{s_{m_0}}\times B(\tau_{m_0}),\ee
with
\be \label{0402-3}\tilde R=\tilde R(\phi,t,\rho)=R^{(m_0)}(\phi,t,\rho)+R^{(m_0)}_+(\phi,t,\rho), \; \,\, (\phi,t,\rho)\in\mathbb T^{d+1}_{s_{m_0}}\times B(\tau_{m_0}).\ee
where $h^{(m_0)},R^{(m_0)}$ and $R^{(m_0)}_+$ obey the conditions {\bf 1(j)-4(j)} with {\bf j}$=m_0$.
Thus,
\be \label{0402-4}||h^{(m_0)}(t,\rho)||_{s_{m_0},\tau_{m_0}}\le C \ve^{-b}, ||\tilde R(\phi,t,\rho)||_{s_{m_0},\tau_{m_0}}\le C \ve ^{A-b}(\log\frac1{\ve})^{C}.\ee
Let $[h^{(m_0)}](I)=\widehat{h^{(m_0)}} (0,I)$ be the  $0$-Fourier coefficient of $h^{(m_0)}(t,I)$ as the function of $t$.
In order to eliminate the dependence of $h^{(m_0)} (t,I)$ on the time-variable $t$, we introduce the following transformation
\[\tilde\Psi:\; \rho=I,\; \phi=\theta+\frac{\p}{\p I}\tilde S(t,I),\; \text{here}\; \tilde S(t,I):=\int_0^t\left( [ h^{(m_0)}](I)-h^{(m_0)}(\xi,I) \right)\, d \xi.\] It is symplectic by  easy verification $d\,\rho\wedge d\phi=d\, I\wedge d\, \theta$.
Note that the transformation is not small. So $\tilde\Psi$ is not close to the Identity. In order to apply $\tilde \Psi$ to $H^{(m_0)}$, we introduce a domain
\[\mathcal{D}:=\left\{t=t_1+t_2\i\in \mathbb T_{s_{m_0}}:\; |t_2|\le \ve^{2b} \right\}\times\left\{I=x+y\,\i\in B(\tau_{m_0}):\; |x-I_0|\le (\log\frac1\ve)^{-C},\, |y|\le\ve^{2b}\right\}.\]
Note that $h^{(m_0)}(t,I)$ is real for real arguments. Thus, for $(t,I)\in\mathcal D$, we have
\[ \left|\Im \frac{\p}{\p I}\tilde S(t,I)\right|=\left|\Im \left(\frac{\p}{\p I}\tilde S(t_1+t_2\i,x+y\i)-\frac{\p}{\p I}\tilde S(t_1,x)\right)\right|\le ||\p^2 \tilde S(t,I) ||_{\mathcal{D}}(|t_2|+|y|)<\ve^{\frac{b}{2}}<\frac{1}{2}s_{m_0}.\]
Therefore,
\[\tilde\Psi(\mathbb T^{d}_{s_{m_0}/2}\times\mathcal D)\subset \mathbb T^{d+1}_{s_{m_0}}\times B(\tau_{m_0}),\]
and
\be\label{040210} \tilde H:=H^{(m_0)}\circ\tilde\Psi=\ve^{-a}H_0(I)+[h^{(m_0)}](I)+\breve{R}(\theta,t,I),\; (\theta,t,I)\in \mathbb T^{d}_{s_{m_0}/2}\times\mathcal D.
 \ee
 where
 \be\label{040211} \breve{R}(\theta,t,I):=\tilde R(\theta+\p_I\tilde S(t,I),t,I).\ee
Recall that $\text{det}\left(\p_I^2\, H_0(I) \right)\ge c_3>0 $.
By {\bf 2(j)} with $j=m_0$, we have
\be\label{ycl1}
\ve^a ||\p_I^2[h^{(m_0)}](I)||_{\tau_{m_0}/2}\le C\ve^{a-b}\left(\log\frac1\ve\right)^C\ll c_3.\ee
 Solving the equation $\p_I H_0(I)+\ve^{a} \p_I [h^{(m_0)}](I)=\omega(I_0)$ by Newton iteration, we get that there exists $I_*\in\mathbb R^d\bigcap B(\tau_{m_0}/2)$ with $|I_*-I_0|\le C\ve^{a-b} (\log\frac1{\ve})^C$ such that
\[\p_I H_0(I_*)+\ve^{a} \p_I [h^{(m_0)}](I_*)=\omega(I_0), \;\text{here}\; \omega(I_0)=\p_I H_0(I_0).\]
Let $I=I_*+\rho$ with $\rho\in\mathbb C^d$ and $|\rho|<\ve^{2b}$. By Taylor formula at $I=I_*$ up to the second order,
\be \label{yyy90}\tilde H=\ve^{-a} H_0(I_*)+[h^{(m_0)}](I_*)+\ve^{-a}\l \omega(I_0), \rho\r+\ve^{-a}\l \Omega \rho,\rho\r+R_{low}(\theta,t,\rho)+R_{high}(\theta,t,\rho)\ee
where
\be \label{yyy91}\Omega=\frac12\p^2_I\left.\left(H_0(I)+\ve^{a}[h^{(m_0)}](I)\right)\right|_{I=I_*} \ee
\be \label{yyy92}R_{low}(\theta,t,\rho)=\breve R(\theta,t,I_*)+\l \p_I \breve R(\theta,t,I_*),\rho\r+\frac12\l \p_I^2 \breve R(\theta,t,I_*)\rho,\rho\r\ee
\be \label{yyy93}R_{high}(\theta,t,\rho)=\ve^{-a}\int_0^1\frac{(1-x)^2}{2}\p^3_x\left(H_0(I_*+x \rho)+\ve^{a}[h^{(m_0)}](I_*+x \rho)+\ve^{a}\breve{ R}(\theta,t,I_*+x \rho)\right)\, dx.\ee

We introduce some notations: With abuse of notation a bit, let $B(r)=\{z\in\mathbb C^d:\; |z-I_*|\le r\}$.
If $f: \mathbb{T}_{s}^{d+1}\times B(r)\rightarrow \mathbb{C}^{l} (l\geq 1)$ is real analytic, and
$$\sup _{x\in\mathbb{T}^{d+1}_{s}\times B(r)}|f(x, I)|\leq C \ve^{\beta}|I|^{\alpha}, \alpha\geq 0,\beta\in\mathbb R,$$
then we write $f=O_{s,r}(\ve^{\beta}|I|^{\alpha}).$ If $f$ is independent of $I$, furthermore, we write $f=O_{s}(\ve^{\beta}|I|^{\alpha}).$ Note that $\ve (\log\frac1{\ve})^C\ll 1$. Let $E=A-9 b-1$.
In view of \eqref{0402-4}, we have
\be\label{yyy94} \p_I^p\breve R(\theta,t,I_*)=O_{\ve^{2b}}(\ve^E),p=0,1,2,\ee
\be\label{yyy95} R_{high}(\theta,t,\rho)=O_{\ve^{2b},\ve^{2b}}(\ve^{-(a+6b)}|\rho|^3 ).\ee
By \eqref{yy4}, \eqref{ycl1} and \eqref{yyy91}, we have
\be\label{ycl2} ||\Omega||\le C,\quad ||\Omega^{-1}||\le C. \ee


\section{Extension of Kolmogorov's Theorem}

By a bit abuse of notation, reset $s_0:=\ve^{2b}$ and $r_0:=\ve^{2b}$.

\begin{thm}\label{thm1}
Consider a Hamiltonian $$H=C+N(I)+R_{low}(\theta, t, I)+R_{high}(\theta, t, I),\; (\theta, t, I)\in \mathbb T^{d+1}_{s_0}\times B(r_0)$$ with the symplectic structure $d I\wedge d\theta$ satisfying
\begin{itemize}
  \item [(1)] $N(I)=\ve^{-a}\langle \omega, I\rangle+\ve^{-a}\langle \Omega I, I\rangle,$
  \item [(2)] $\omega=\omega(I_0)$ is D.C., that is
  \[|\ve^{-a}\langle k, \omega\rangle+l|\geq \frac{\gamma}{(1+|k|)^{d+1}}, (k,l)\in\mathbb Z^d\times\mathbb Z\setminus\{(0,0)\},\]
  \item [(3)] $||\Omega||\le C,\; ||\Omega^{-1}||\le C,$

  \item [(4)]  $N(I),$ $R_{low}(\theta, t, I)$ and $R_{high}(\theta, t, I)$ are real analytic in $\mathbb{T}^{d+1}_{s_{0}}\times B(r_{0}),$
    \item [(5)]
        $R_{low}=R_{0}(\theta, t)+\langle R_{1}(\theta, t), I\rangle+\langle R_{2}(\theta, t)I, I\rangle,$ with
  $ R_j=O_{s_0}(\ve^{E}),$ $j=0,1,2$,
\item [(6)] $R_{high}=O_{s_0,r_0}(\ve^{-(a+6b)}|I|^{3}).$
\end{itemize}
Then there exists a symplectic coordinate change
\[\Psi:\; I=\rho+u(\phi,t,\rho),  \theta=\phi+v(\phi,t,\rho), t=t\]
 where $u=O_{s_0/2,r_0/2}(\ve^{\frac{E}{2}}),v=O_{s_0/2,r_0/2}(\ve^{\frac{E}{2}})$
 such that
 $\Psi\left(\mathbb{T}^{d+1}_{s_{0}/2}\times B(\frac{r_{0}}{2})\right)\subset \mathbb{T}^{d+1}_{s_{0}}\times B(r_{0})$ and
the Hamiltonian $H$ is changed by $\Psi$ into
$$H^{\infty}(\phi, t, \rho)=H\circ \Psi(\phi, t, \rho)=C+\ve^{-a}\langle \omega, \rho\rangle+\ve^{-a}\langle \Omega^{\infty}\rho, \rho\rangle+O_{s_{0}/2,r_{0}/2}(\ve^{-a-6b}|\rho|^{3}),(\phi, t, \rho)\in \mathbb{T}^{d+1}_{s_{0}/2}\times B(\frac{r_{0}}{2}),$$
where $$|\Omega^{\infty}-\Omega|\leq C\ve^{\frac{E}{2}}.$$
\end{thm}
\begin{cor}
The Hamiltonian system defined by Hamiltonian $H$ has an invariant torus with rotational frequency $(\ve^{-a}\omega(I_0),2\pi)$ in the extended phase space $\mathbb T^{d+1}\times \mathbb R^d$.
\end{cor}
The proof is finished by the following iterative lemma. To this end, we introduce some iterative constants, iterative parameters and iterative domains:
\begin{itemize}
  \item $\epsilon_{0}=\ve^{E}, \epsilon_{j}=(\ve^{E})^{(4/3)^{j}}, j=1, 2, \cdots, \epsilon_{0}>\epsilon_{1}>\cdots >\epsilon_{j}\downarrow 0$;
  \item $e_{j}=\frac{\sum_{l=1}^{j}l^{-2}}{100\sum_{l=1}^{\infty}l^{-2}}$;
  \item $s_{j}=s_{0}(1-e_{j})$ (so $s_{j}>s_{0}/2$ for all $j=1, 2, \cdots$);
\end{itemize}
\begin{itemize}
  \item $r_{j}=r_{0}(1-e_{j})$ (so $r_{j}>r_{0}/2$ for all $j=1, 2, \cdots$);
  \item $s_{j}^{l}=(1-\frac{l}{10})s_{j}+\frac{l}{10}s_{j+1}$ (so $s_{j+1}<s_{j}^{l}<s_{j}, l=1, 2,\cdots, 10$);
  \item $r_{j}^{l}=(1-\frac{l}{10})r_{j}+\frac{l}{10}r_{j+1}$ (so $r_{j+1}<r_{j}^{l}<r_{j}, l=1, 2,\cdots, 10$);
  \item For any $\alpha>0$, write $\e_j^{\alpha-}= j^C\,\left( \log\frac1\ve\right)^C\,\e^{\alpha}_j$.
\end{itemize}
\begin{lem}[Iterative Lemma]
Assume that we have had $m$ Hamiltonian functions, for $j=0, 1, \cdots, m-1$ satisfying
\begin{eqnarray}
\label{eq61} && H^{(j)}=N^{(j)}+R_{low}^{(j)}+R_{high}^{(j)},\\
\label{eq62} && N^{(j)}=\ve^{-a}\langle \omega, I\rangle+\ve^{-a}\langle \Omega^{(j)}I, I\rangle,
                \parallel \Omega^{(j)}-\Omega^{(j-1)}\parallel\leq C\epsilon_{j-1}\ve^{-b^\prime},j\ge 1, \Omega^{(0)}=\Omega,b^\prime:=12d(a+b),  \\
\label{eq63} && R_{low}^{(j)}(\theta,t,I) , R_{high}^{(j)}(\theta,t,I) \text{ are real analytic in } \mathbb{T}^{d+1}_{s_{j}}\times B(r_{j}), \\
\label{eq64} && R^{(j)}_{low}=R^{(j)}_{0}(\theta, t)+\langle R^{(j)}_{1}(\theta, t), I\rangle +\langle R^{(j)}_{2}(\theta, t)I, I\rangle, \\
\label{eq65} &&  R^{(j)}_{p}(\theta,t,I)=O_{s_{j}}(\epsilon_{j}),\; (p=0, 1, 2),\\
\label{eq66} && R_{high}^{(j)}(\theta,t,I)=O_{s_{j}, r_{j}}(\ve^{-(a+6b)}|I|^{3}).
\end{eqnarray}
Then there exists a symplectic coordinate
\be\label{eq66x}\Psi_m:\; I=\nu+\rho+u_{m}(\phi, t, \rho), \theta=\phi+v_{m}(\phi, t, \rho), t=t \ee
where $\nu$ is a constant depending $\ve$, $(\phi, t, \rho)\in \mathbb{T}^{d+1}_{s_{m}}\times B(r_{m})$ and
\begin{equation}\label{eq66y}
 u_{m}=O_{s_{m}, r_{m}}(\ve^{-3b^\prime}\e_{m-1}),v_{m}=O_{s_{m}, r_{m}}(\ve^{-3b^\prime}\e_{m-1}),
\end{equation}
\begin{equation}\label{66++}
\Psi_{m}(\mathbb{T}^{d+1}_{s_{m}}\times B(r_{m}))\subset \mathbb{T}^{d}_{s_{m-1}}\times B(r_{m-1})
\end{equation}
such that the changed Hamiltonian
$H^{(m)}=H^{(m-1)}\circ \Psi_{m}$ satisfies \eqref{eq61}-\eqref{eq66} by replacing $j$ by $m.$
\end{lem}
\begin{proof}
Assume the coordinate change $\Psi=\Psi_{m}$ can be implicitly defined by
\begin{equation}\label{71} \Psi:\;
    I=\nu+\rho+\frac{\partial S}{\partial \theta},
    \phi=\theta+\frac{\partial S}{\partial \rho}, t=t,
\end{equation}
where function  $S=S(\theta, t, \rho)$ and constant vector $\nu$ will be specified by homological equations later. Then $\Psi$ is symplectic if it is well-defined.
Let
\begin{equation}\label{72}
S=S_{0}(\theta, t)+\langle S_{1}(\theta, t), \rho\rangle +\langle S_{2}(\theta, t)\rho, \rho\rangle.
\end{equation}
Then
\begin{eqnarray*}
H^{(m)}&=&H^{(m-1)}\circ \Psi_{m}\label{326}\\
&=& N^{(m-1)}(v+\rho+\frac{\partial S}{\partial \theta})+R_{low}^{(m-1)}(\theta, t, v+\rho+\frac{\partial S}{\partial \theta})+
R_{high}^{(m-1)}(\theta, t, v+\rho+\frac{\partial S}{\partial \theta})+\frac{\partial S(\theta, t, \rho)}{\partial t}\label{327}\\
&=&\ve^{-a}\langle \omega, v+\rho+\partial_{\theta}S\rangle+\ve^{-a}\langle \Omega^{(m-1)}(v+\rho+\partial_{\theta}S), v+\rho+\partial_{\theta}S\rangle
+R_{0}^{(m-1)}(\theta, t)+\frac{\partial S}{\partial t}\label{328}\\
&&+\langle R_{1}^{(m-1)}(\theta, t), v+\rho+\partial_{\theta}S\rangle+\langle R_{2}^{(m-1)}(\theta, t)(v+\rho+\partial_{\theta}S), v+\rho+\partial_{\theta}S\rangle
+R_{high}(\theta, t, v+\rho+\partial_{\theta}S).\label{329}
\end{eqnarray*}
{ Let $\omega\cdot \partial_{\theta} S=\langle \omega, \partial_{\theta} S\rangle.$ By a simple computation, we have}
\begin{eqnarray}
H^{(m)}&=&\ve^{-a}\langle \omega, v\rangle+\ve^{-a}\langle \omega, \rho\rangle
+\ve^{-a}\left(\omega\cdot \partial_{\theta}S_{0}+\langle \omega\cdot \partial_{\theta}S_{1}, \rho\rangle+\langle \omega\cdot \partial_{\theta}S_{2}\rho, \rho\rangle\right)\nonumber\\
&&+\ve^{-a}\left(\langle \Omega^{(m-1)}v, v\rangle+2\langle \Omega^{(m-1)}v, \rho\rangle+ \langle \Omega^{(m-1)}\rho, \rho\rangle
+2\langle \Omega^{(m-1)}v, \partial_{\theta} S\rangle\right)\nonumber\\
\label{eq*}&&+\ve^{-a}\left(2\langle \Omega^{(m-1)}\rho, \partial_{\theta}S_{0}\rangle+2\langle \Omega^{(m-1)}\rho, \langle \partial_{\theta}S_{1}, \rho\rangle\rangle
+2\langle \Omega^{(m-1)}\rho, \langle \partial_{\theta}S_{2}\rho, \rho\rangle\rangle\right)\\
&&+\ve^{-a} \langle \Omega^{(m-1)}\p_\theta S, \p_\theta S\rangle+R_{0}^{(m-1)}+\langle R_{1}^{(m-1)}, v+\partial_{\theta}S\rangle+\langle R_{1}^{(m-1)}, \rho\rangle\nonumber\\
&&+\langle R_{2}^{(m-1)}(v+\partial_{\theta}S), v+\partial_{\theta}S\rangle+\langle R_{2}^{(m-1)}\rho, \rho\rangle+2\langle R_{2}^{(m-1)}(v+\partial_{\theta}S), \rho\rangle\nonumber\\
&&+\frac{1}{2}\langle (\partial_{I}^{3}R_{high}(\theta, t, 0)\partial_{\theta}S_{0})\rho, \rho\rangle+\partial_{t}S\nonumber\\ && +R_{high}(\theta, t, v+\rho+\partial_{\theta} S)-\frac{1}{2}\langle (\partial_{I}^{3}R_{high}(\theta, t, 0)\partial_{\theta}S_{0})\rho, \rho\rangle.\nonumber
\end{eqnarray}
In the following, we omit the term $\ve^{-a}(\langle \omega, v\rangle+\langle \Omega^{(m-1)}v, v\rangle)$ which does not affect the dynamics. Let
\be\label{327-1}R_*(\theta,t)=2\ve^{-a}\Omega^{(m-1)}\partial_{\theta}S_{0}(\theta,t)+R_{1}^{(m-1)}(\theta,t),\ee
\be\label{327-2}R_{**}(\theta,t)=\frac{1}{2}\partial^{3}_{I}R_{high}(\theta, t, 0)\partial_{\theta}S_{0}(\theta,t)+2\ve^{-a}\Omega^{(m-1)}(\partial_{\theta}S_{1}(\theta,t))+R_2^{(m-1)}(\theta,t).\ee
From the changed $H^{(m)}$, we derive the homological equations:
\begin{eqnarray}
\label{eqh1}&&\ve^{-a}\omega\cdot \partial_{\theta}S_{0}(\theta, t)+\partial_{t}S_{0}(\theta, t)+R_{0}^{(m-1)}(\theta, t)=\widehat{R_{0}^{(m-1)}}(0, 0),\\
\label{eqh2}&& \ve^{-a}\omega\cdot \partial_{\theta}S_{1}(\theta, t)+\partial_{t}S_{1}(\theta, t)+R_*(\theta,t)=\widehat{R_*}(0,0),\\
\label{eqh3}&& 2\ve^{-a}\Omega^{(m-1)}v=-\widehat{R_*}(0,0),\\
\label{eqh4}&& \ve^{-a}\omega\cdot \partial_{\theta}S_{2}(\theta, t)+\partial_{t}S_{2}(\theta, t)
+R_{**}(\theta,t)=\widehat{R_{**}}(0,0).
\end{eqnarray}

Let
\begin{equation}\label{eq81}\Omega^{(m)}=\Omega^{(m-1)}+\ve^{a}\widehat{R_{**}}(0,0),\end{equation}
\begin{eqnarray}
\label{eq82}R_{+}&=& \ve^{-a}\left(\langle \Omega^{(m-1)}\partial_{\theta}S, \p_\theta S\rangle+2\langle \Omega^{(m-1)} \nu, \partial_{\theta}S\rangle\right)\\
\label{eq83}&&+\langle R_{1}^{(m-1)}, \nu+\partial_{\theta}S\rangle\\
\label{eq84}&& +\langle R_{2}^{(m-1)}(\nu+\partial_{\theta}S), \nu+\partial_{\theta}S\rangle+2\langle R_{2}^{(m-1)}(\nu+\partial_{\theta}S), \rho\rangle\\
\label{eq85}&& +R_{high}(\theta, t, v+\rho+\partial_{\theta}S)-\frac{1}{2}\langle (\partial_{I}^{3}R_{high}(\theta, t, 0)\partial_{\theta}S_{0})\rho, \rho\rangle\\
\label{eq86}&& +2\ve^{-a}\langle \Omega^{(m-1)}\rho, \langle \partial_{\theta}S_{2}\rho, \rho\rangle\rangle.
\end{eqnarray}
Then
\begin{equation}\label{eq87}
H^{(m)}=C+\ve^{-a}(\langle \omega, \rho\rangle+\ve^{-a}\langle \Omega^{(m)}\rho, \rho\rangle)+R_{+}.
\end{equation}

$\bullet$ Soluitons of the homological equations
\begin{itemize}
  \item [(1)] Solution to \eqref{eqh1}: Passing to Fourier coefficients,
$${S}_{0}(\theta,t)=\sum_{(0,0)\neq(k,l)\in\mathbb Z^d\times\mathbb Z}\frac{\i}{\ve^{-a}\langle k,\omega\rangle+l}\widehat{R^{(m-1)}_{0}}(k,l)e^{\i(\langle k,\theta\rangle+lt)}.$$
By $\omega\in D.C.,$
\begin{equation}\label{eq91}
\parallel S_{0}\parallel_{s_{m-1}^{1}}\leq \frac{C}{(s_{m-1}-s_{m-1}^{1})^{2d+2}}\frac{1}{\gamma}\parallel R_{0}^{(m-1)}\parallel_{s_{m-1}}\leq \e^{1-}_{m-1}\ve^{-b(4d+4)}.
\end{equation}
  \item [(2)] Solution to \eqref{eqh2}: By \eqref{eq65}, \eqref{327-1} and \eqref{eq91}  with $j=m-1$, and using Cauchy's estimate, we have
\be R_{*}=O_{s_{m-1}^2}(\e_{m-1}^{1-}\ve^{-b(4d+6)-a}).\label{327-3}\ee

By \eqref{eqh2}, we have $${S}_{1}(\theta,t)=\sum_{(0,0)\neq(k,l)\in\mathbb Z^d\times\mathbb Z}\frac{\i}{\ve^{-a}\langle k,\omega\rangle+l}\widehat{R_{*}}(k,l)e^{\i(\langle k,\theta\rangle+lt)}.$$
Therefore
\begin{equation}\label{eq92}
\parallel S_{1}\parallel_{s_{m-1}^{3}}\leq \frac{C}{(s_{m-1}^{2}-s_{m-1}^{3})^{2d+2}}\frac{1}{\gamma}\parallel R_{*}\parallel_{s_{m-1}^{2}}\leq \e^{1-}_{m-1}\ve^{-b(8d+10)-a}.
\end{equation}
  \item [(3)] Solution to \eqref{eqh3}: By \eqref{eqh3} and \eqref{327-3},
\begin{equation}\label{eq93}
|\nu|=\left|(2\ve^{-a}\Omega^{(m-1)})^{-1}\widehat{R_*}(0,0) \right|\le C \e_{m-1}^{1-}\ve^{-b(4d+6)}.
\end{equation}
 \item [(4)] Similarly, by \eqref{327-2}, \eqref{eq91}, \eqref{eq92}  and Cauchy's estimate, we have
$$\parallel R_{**}(\theta,t)\parallel_{s^{4}_{m-1}}\leq C \e^{1-}_{m-1}\ve^{-b(8d+12)-2a}.$$
Passing \eqref{eqh4} to Fourier coefficients,
$${S}_{2}(\theta,t)=\sum_{(0,0)\neq(k,l)\in\mathbb Z^d\times\mathbb Z}\frac{\i}{\ve^{-a}\langle k,\omega\rangle+l}\widehat{R_{**}}(k,l)e^{\i(\langle k,\theta\rangle+lt)}.$$
It follows
\begin{equation}\label{eq95}
\parallel S_{2}\parallel_{s_{m-1}^{5}}\leq \frac{C}{(s_{m-1}^{4}-s_{m-1}^{5})^{2d+2}}\frac{1}{\gamma}\parallel R_{**}\parallel_{s_{m-1}^{4}}\leq \e^{1-}_{m-1}\ve^{-b(12d+16)-2a}.
\end{equation}
Now
\begin{equation}\label{eq96}
|\Omega^{(m)}-\Omega^{(m-1)}|=\ve^{a}|\widehat{R_{**}}(0, 0)|\leq \ve^{-b^\prime}\e_{m-1}.
\end{equation}
This verifies \eqref{eq62} in the Iterative Lemma.
By \eqref{eq91}, \eqref{eq92} and \eqref{eq95}, we have
\begin{equation}\label{eq100}
\parallel S\parallel_{s_{m-1}^{5}, r_{m-1}}=\parallel S_{0}+\langle S_{1}, \rho\rangle+\langle S_{2}\rho, \rho\rangle\parallel_{s_{m-1}^{5}, r_{m-1}}\leq C\ve^{-\frac{3b^\prime}{2}}\e_{m-1}.
\end{equation}
By Cauchy's estimate,
\begin{equation}\label{eq101}
\parallel \sum_{|\alpha|+|\beta|\leq 2}\partial_{\theta}^{\alpha}\partial_{\rho}^{\beta}S\parallel_{s_{m-1}^{6}, r_{m-1}^{1}}\leq C\ve^{-2b^\prime}\e_{m-1}.
\end{equation}
By using the implicit theorem, we have that there exist $u=u_{m}(\phi, t, \rho),$ $v=v_{m}(\phi, t, \rho)$ with
\begin{equation}\label{eq102}
\parallel u\parallel_{s_{m-1}^{7}, r_{m-1}^{7}}, \parallel v\parallel_{s_{m-1}^{7}, r_{m-1}^{7}}\leq C\ve^{-3b^\prime}\e_{m-1}
\end{equation}
such that $\Psi=\Psi_{m}$
\begin{equation}\label{eq103}
\Psi_{m}: \left\{
              \begin{array}{ll}
                I=v+\rho+u(\phi, t, \rho),\\
                \theta=\phi+v(\phi, t, \rho),\\ t=t
              \end{array}
            \right.
 (\phi, t, \rho)\in \mathbb{T}_{s_{m-1}^{7}}^{d+1}\times B(r_{m-1}^{7})
\end{equation} is well-defined and
$$\Psi_{m}(\mathbb{T}_{s_{m}}^{d+1}\times B(r_{m}))\subset \mathbb{T}_{s_{m-1}}^{d+1}\times B(r_{m-1}).$$
This verifies \eqref{eq66y} and \eqref{66++}.
\end{itemize}
We are now in position to estimate the new perturbation $R_{+}.$  By \eqref{eq93}, \eqref{eq101}, \eqref{eq102} and \eqref{eq103},
we have
\begin{eqnarray}\label{eq111}
\parallel \eqref{eq82}\parallel_{s_{m-1}^7,r_{m-1}^7} &\leq &C\ve^{-a}(\ve^{-3b^\prime}\e_{m-1}^{1-})^{2}\leq C\ve^{-a-6b^\prime}\e_{m-1}^{2}\left( \log\frac1\ve\right)^C
\le\e_m \ve^{4b^\prime}.
\end{eqnarray}
Similarly,
\begin{equation}\label{eq112}
\parallel \eqref{eq83}\parallel_{s_{m-1}^7,r_{m-1}^7}\leq \e_{m}\ve^{4b^\prime},
\end{equation}
\begin{equation}\label{eq113}
\parallel \eqref{eq84}\parallel_{s_{m-1}^7,r_{m-1}^7}\leq \e_{m}\ve^{4b^\prime}.
\end{equation}
Clearly,
\begin{equation}\label{eq114}
\eqref{eq86}=O_{s_{m-1}^7,r_{m-1}^7}(\ve^{-4b^\prime}\e_{m-1}^{1-}|\rho|^{3})=O_{s_{m-1}^7,r_{m-1}^7}(|\rho|^{3}).
\end{equation}
Recall $R_{high}=O_{s_{m-1},r_{m-1}}(\ve^{-(a+6b)}|I|^{3}).$ 
Then by Taylor formula and in view of \eqref{eq100}, we have
\begin{eqnarray}
\eqref{eq85}&=&R_{high}(\theta, t, \rho+\partial_{\theta} S)-\langle (\partial_{I}^{3}R_{high}(\theta, t, 0)\partial_{\theta}S_{0})\rho, \rho\rangle\nonumber\\
&=& \sum_{|\alpha|\ge3,\alpha\in\mathbb Z^d_+}\frac{1}{\alpha!}\, \left((\rho+\p_\theta\, S)\cdot \p_I \right)^{\alpha}R_{high}(\theta,t,0) -\frac{1}{2}\langle (\partial_{I}^{3}R_{high}(\theta, t, 0)\partial_{\theta}S_{0})\rho, \rho\rangle \nonumber\\
&=& O_{s_{m-1}^7,r_{m-1}^7}(\ve^{-(a+6b)} |\rho|^{3})
+O_{s_{m-1}^7,r_{m-1}^7}(\e_m \ve^{3 b^\prime}).\label{327-9}
\end{eqnarray}
Consequently, by \eqref{eq111}-\eqref{327-9}, we have
\be\label{327-10}R_+(\phi,t,\rho)=O_{s_{m-1}^7,r_{m-1}^7}(\ve^{-(a+6b)} |\rho|^{3})+O_{s_{m-1}^7,r_{m-1}^7}(\e_m \ve^{3b^\prime}). \ee
By developing $R_+(\phi,t,\rho)$ into Taylor formula of order $3$ in $\rho$, we can re-write
\[R_+(\phi,t,\rho)=R^{(m)}_{low}(\phi,t,\rho)+R^{(m)}_{high}(\phi,t,\rho)\]
where
$$R^{(m)}_{low}=R_{0}^{(m)}(\phi, t)+\langle R_{1}^{(m)}(\phi, t), \rho\rangle+\langle R_{2}^{(m)}(\phi, t)\,\rho,\rho\rangle, $$
 and $R_{p}^{(m)}(\phi, t)=O_{s_m,r_m}(\e_m)$ ($p=0,1,2$) and $R^{(m)}_{high}(\phi,t,\rho) =O_{s_{m}, r_{m}}(\ve^{-(a+6b)}|\rho|^{3})$.
This verifies  \eqref{eq64}-\eqref{eq66} with $j=m$. As for the functions $R^{(m)}_{low}(\phi,t,\rho)$ and $R^{(m)}_{high}(\phi,t,\rho)$ is real for the real argument $(\phi,t,\rho)\in \mathbb T^{d+1}_{s_{m}}\times B(r_m)$, the proof is the same as that in the proof of Lemma \ref{lemma2.1}.
This completes the proof of the Iterative Lemma.
\end{proof}
\section{Proof of the extended Kolmogorov's Theorem}
Let $\Psi(x, I)=\lim_{m\rightarrow \infty}\Psi_{1}\circ \cdots\circ\Psi_{m}(x,I), (x, I)\in \mathbb{T}^{d+1}_{s_{0}/2}\times B(r_{0}/2).$
Recall that in the Iterative Lemma, we have proved:
\begin{itemize}
  \item [(i)] $\mathbb{T}_{s_{0}}^{d+1}\times B(r_{0})\supset \cdots\supset \mathbb{T}_{s_{j}}^{d+1}\times B(r_{j})\supset \cdots\supset \mathbb{T}_{s_{0}/2}^{d+1}\times B(r_{0}/2),$
  \item [(ii)] $\Psi_{j}: \mathbb{T}^{d+1}_{s_{j}}\times B(r_{j})\rightarrow \mathbb{T}^{d+1}_{s_{j-1}}\times B(r_{j-1})$ is real analytic.
\end{itemize}
Then $\Psi^{(m)}:=\Psi_{1}\circ \cdots\circ\Psi_{m}: \mathbb{T}_{s_{0}/2}^{d+1}\times B(r_{0}/2)\subset \mathbb{T}_{s_{m}}^{d+1}\times B(r_{m})\rightarrow \mathbb{T}_{s_{0}}^{d+1}\times B(r_{0})$ is well-defined and analytic in $\mathbb{T}_{s_m}^{d+1}\times B(r_m).$
Also by \eqref{eq102} and \eqref{eq103} in iterative lemma, we have
\begin{equation}\label{eq201}
\parallel (\partial\Psi_{j}(x, I)-Id)(z)\parallel_{s_j, r_j}\leq \ve^{-4b^\prime}\epsilon_{j-1}\parallel z\parallel,
\end{equation}
for $\forall z\in T_{(x, I)}(\mathbb{T}_{s_j}^{d+1}\times B(r_j))=\mathbb{C}^{d+1}\times \mathbb{C}^{d},$ where $\partial\Psi_{j}$ is the tangent map of $\Psi_{j},$
$T_{(x, I)}(\mathbb{T}_{s_j}^{d+1}\times B(r_j))$ is the tangent space at point $(x, I)\in \mathbb{T}_{s_j}^{d+1}\times B(r_j).$
Observe that
\begin{equation}\label{eq202}
\partial\Psi^{(m)}=(\partial\Psi_{1}\circ \Psi_{2}\circ \cdots\circ \Psi_{m})(\partial\Psi_{2}\circ \Psi_{3}\circ\cdots\circ\Psi_{m})\cdots(\partial\Psi_{m}).
\end{equation}
Then by \eqref{eq201},
\begin{equation}\label{eq203}
\parallel \partial\Psi^{(m)}(x, I)z\parallel_{s_m, r_m}\leq \Pi_{j=1}^{m}(1+\e_{j-1}\ve^{-4b^\prime})\parallel z\parallel
\leq \Pi_{j=1}^{m}(1+\frac{1}{2^{j}})\parallel z\parallel\leq 2\parallel z\parallel, \forall z\in \mathbb{C}^{d+1}\times \mathbb{C}^{d}.
\end{equation}
Thus, for $w=(x, I)\in \mathbb{T}^{d+1}_{s_{0}/2}\times B(r_{0}/2),$
\begin{eqnarray}
&&|\Psi^{(m+1)}(w)-\Psi^{(m)}(w)|\nonumber\\
&=&|\Psi^{(m)}(\Psi_{m+1}(w))-\Psi^{(m)}(w)|\nonumber\\
\label{eq204}&\leq &\parallel \partial \Psi^{(m)}(\Psi_{m+1}(w))\parallel_{s_{0}/2, r_{0}/2}\sup_{w\in \mathbb{T}^{d+1}_{s_{0}/2}\times B(r_{0}/2)}
\parallel \Psi_{m+1}(w)-w\parallel\\
&\leq & 2\ve^{-4b^\prime}\e_{m}\leq \e_{m}^{1/2},\nonumber
\end{eqnarray}
where we have used \eqref{eq66y}, \eqref{eq93} and \eqref{eq203} in \eqref{eq204}.
Write $\Psi^{(m)}=\Psi^{(1)}+\sum_{l=2}^{m}(\Psi^{(l)}-\Psi^{(l-1)}), $ here $\Psi^{(1)}:=\Psi_{1}.$ Then
\begin{eqnarray*}
\parallel \Psi(w)-w\parallel_{s_{0}/2, r_{0}/2}
&\leq &\parallel \Psi_{1}(w)-w\parallel_{s_{0}/2, r_{0}/2}+\sum_{l=2}^{\infty}\parallel \Psi^{(l)}(w)-\Psi^{(l-1)}(w)\parallel_{s_{0}/2, r_{0}/2}\le \sum_{l=0}^{\infty}\e_{l}^{1/2}\leq  \ve.
\end{eqnarray*}
This completes the proof of theorem.
\section{Application to coupled Duffing oscillators}\label{section5}
Consider the coupled Duffing oscillators
\begin{equation}\label{eq5.1}
\ddot{x_{j}}+x_{j}^{2n+1}+\frac{\partial F(x, t)}{\partial x_{j}}=0, \;\;j=1, 2, \cdots, m, x=(x_1,...,x_m)
\end{equation}
where $m, n>0$ are fixed integers, and
\begin{equation}\label{eq5.2}
F(x, t)=\sum_{0\leq |\alpha|\leq 2n+1}P_{\alpha}(t)x^{\alpha}=\sum_{0\leq \alpha_{1}+\cdots+\alpha_{m}\leq 2n+1}P_{\alpha_{1}\cdots\alpha_{m}}(t)x_{1}^{\alpha_{1}}\cdots x_{m}^{\alpha_{m}}, \alpha_{j}\in\mathbb{Z}_{+}.
\end{equation}
and $P_{\alpha}(t): \mathbb{T}\rightarrow \mathbb{R}$ is analytic (in this section, $\mathbb T:=\mathbb R/\mathbb Z$).

Let $\widetilde{A}$ be a large constant. Replacing $x_{j}$ by $\widetilde{A}x_{j}$ in \eqref{eq5.1}, we get
$$\widetilde{A}\ddot{x}_{j}+\widetilde{A}^{2n+1}x_{j}^{2n+1}+\widetilde{A}^{-1}\frac{\partial F(\widetilde{A}x, t)}{\partial x_{j}}=0,$$
that is,
\begin{equation}\label{eq5.4}
\ddot{x_{j}}+\widetilde A^{2n}x_{j}^{2n+1}+\widetilde{A}^{-2}\frac{\partial F(\widetilde{A}x, t)}{\partial x_{j}}=0, j=1, 2, \cdots,m.
\end{equation}
Let
$$y_{j}=\widetilde A^{-n}\dot{x}_{j}, j=1, 2, \cdots, m.$$
Then
$$\dot{y_{j}}=\widetilde{A}^{-n}\ddot{x_{j}}=-\widetilde{A}^{n}x^{2n+1}_{j}-\widetilde{A}^{-(n+2)}\frac{\partial F(\widetilde{A}x, t)}{\partial x_{j}}.$$
Thus, \eqref{eq5.1} can be written as a Hamiltonian system
\begin{equation}\label{eq5.5}
\dot{x_{j}}=\frac{\partial H}{\partial y_{j}}, \dot{y_{j}}=-\frac{\partial H}{\partial x_{j}}, j=1, 2, \cdots, m,
\end{equation}
where
\begin{equation}\label{eq5.6}
H= \widetilde{A}^{n}\sum_{j=1}^{m}\left(\frac{1}{2}y_{j}^{2}+\frac{1}{2(n+1)}x_{j}^{2n+2}\right)+\widetilde{A}^{-(n+2)}F(\widetilde{A}x, t).
\end{equation}
Consider an auxiliary Hamiltonian system which is autonomous:
\begin{equation}\label{eq5.7}
\dot{x}=\frac{\partial H_{0}}{\partial y}, \dot{y}=-\frac{\partial H_{0}}{\partial x}, H_{0}=\frac{1}{2}y^{2}+\frac{1}{2(n+1)}x^{2n+2}, (x, y)\in \mathbb{R}^{2}.
\end{equation}
Let $(x, y)=(u_{0}(t), v_{0}(t))$ be the solution to \eqref{eq5.7} with initial values $(u_{0}(0),v_{0}(0))=(1, 0).$ Then this solution $(u_{0}, v_{0})$ is clearly periodic.
Let $T_{0}$ be its minimal period. By energy conservation, we have
\begin{equation}\label{eq5.8}
(n+1)v^{2}_{0}(t)+u_{0}^{2n+2}(t)=1, \forall t\in \mathbb{R}.
\end{equation}
Let
$$ \left\{
     \begin{array}{ll}
       x=c^{\alpha}I^{\alpha} u_{0}(\theta T_{0}),\\
       y=c^{\beta}I^{\beta}v_{0}(\theta T_{0}),
     \end{array}
   \right.
(\theta, I)\in \mathbb{T}\times \mathbb{R}_+,$$
where $\alpha=\frac{1}{n+2}, \beta=1-\alpha, c=\frac{1}{\alpha T_{0}}.$ It is easy to check $\det \frac{\partial(x, y)}{\partial(\theta, I)}=1.$ That is,
$d x\wedge d y=d\theta\wedge d I.$ Thus
\begin{equation}\label{eq5.9}
\Phi: \left\{
        \begin{array}{ll}
          x_{j}=c^{\alpha}I^{\alpha} _{j}u_{0}(\theta_{j}T_{0}),\\
          y_{j}=c^{\beta}I^{\beta} _{j}v_{0}(\theta_{j}T_{0}),
        \end{array}
      \right.
j=1, 2, \cdots,m.
\end{equation}
is well-defined and symplectic transformation $\Phi: (\theta, I)\in \mathbb{T}^{m}\times \mathbb{R}_+^{m}\rightarrow \mathbb{R}^{m}\times \mathbb{R}^{m}.$
Moreover,
$$H^{(1)}:=H\circ \Phi =\widetilde{H}_{0}(I)+\widetilde{R}(\theta, t, I),$$
where
$$\widetilde{H}_{0}(I)=\frac{c^{2\beta}}{2(n+1)}\widetilde{A}^{n}\sum_{j=1}^{m}I_{j}^{\frac{2(n+1)}{n+2}},$$
$$ \widetilde{R}(\theta, t, I)=\widetilde{A}^{-(n+2)}F(\widetilde{A}c^{\alpha}I_{1}^{\alpha}u_{0}(\theta_{1}T_{0}), \cdots, \widetilde{A}c^{\alpha}I_{m}^{\alpha}u_{0}(\theta_{m}T_{0}), t).$$
Note that $F$ is a polynomial in $x\in \mathbb{R}^{m}$ with degree $2n+1.$ So $\widetilde{R}(\theta, t, I)$ is a polynomial of $\widetilde A$ of degree $n-1$ with its coefficients analytically depending on  $(\theta, t, I)\in \mathbb{T}^{m+1}\times [1, 2].$ Let $\ve=\widetilde{A}^{-1}, a=n, b=n-1,$ and
$$H_{0}(I)=\frac{c^{2\beta}}{2(n+1)}\sum_{j=1}^{m}I_{j}^{\frac{2(n+1)}{n+2}}, R(\theta, t, I)=\ve^{b}\widetilde{R}(\theta, t, I).$$
Then
$$H^{(1)}=\ve^{-a}H_{0}(I)+\ve^{-b}R(\theta, t, I).$$ where $R(\theta, t, I)$ depends on $\ve$. Note that it is harmless not to write explicitly the dependence  of $R$ on $\ve$.
Applying Theorem \ref{thm01} and Remark \ref{RMK2} to $H^{(1)}$, we have that there exists a subset $\tilde{J}_{0}(\widetilde A)\subset [1, 2]^{m}$ with $\text{Leb }\,\tilde{J}_{0}(\widetilde A)\geq 1-(\log\widetilde  A)^{-c_{0}}$ (some $c_{0}>0$) such that for any solution to
\eqref{eq5.5} starting from
 $I(0)=(I_1(0),\cdots,I_m(0))\in \tilde{J}_{0}(\widetilde A)$ is quasi-periodic with rotational frequency $(\omega(I(0)), 1)$,
 where $\omega=(\omega_j:\, j=1,...,m)$ and $\omega_{j}=\ve^{-a}\frac{\partial H_{0}(I(0))}{\partial I_{j}}.$ Returning to the coordinates $(x_i,\dot x_i: i=1,...,m)$, we have that for any large $\widetilde A$ there exist sets $\Theta_{\widetilde A}$ and $\tilde\Theta_{\widetilde A}$ with
 $\Theta_{\widetilde A}\subset\tilde\Theta_{\widetilde A}\subset[C_1\widetilde A,C_2\widetilde A]^{2m}$
 and
 $\text{Leb}\,\tilde\Theta_{\widetilde A}=1,\; \text{Leb}\,\Theta_{\widetilde A}\ge 1-\left(\log \widetilde A\right)^{-C_0}$
 such that for any solution to
\eqref{eq5.1} starting from  $(x_j(0),\dot x_j(0): j=1,...,m)\in \Theta_{\widetilde A}$  is quasi-periodic with rotational frequency $(\omega(I(0)), 1)$. Note that any quasi-periodic solution is bounded. Also note that
$$\lim_{\widetilde A\to\infty}\frac{\tilde \Theta_{\widetilde A}}{\Theta_{\widetilde A}}=1$$
Thus, we see that \eqref{eq5.1} is almost Lagrangian stability. This completes the proof of Theorem \ref{thm02}.

\section*{Acknowledgements}
\noindent


{\it The work was supported in part by National Natural Science Foundation of China
(Nos. 11771093, 12071254).}


\end{document}